\documentclass[reqno]{amsart}

\usepackage{amsmath}
\usepackage{amsthm}
\usepackage{amscd}
\usepackage{bbm}
\usepackage{enumerate}
\usepackage{amssymb}
\usepackage{palatino}
\usepackage{amscd}
\usepackage{verbatim}
\usepackage{mathrsfs}
\usepackage[margin=1.35in]{geometry}

\numberwithin{equation}{section}

\subjclass[2010]{05B15, 05A16.}

\keywords{Alon-Tarsi conjecture, Latin squares, equidistribution, integration over unitary groups}

\title{Square-root cancellation for the signs of Latin squares}
\author{Levent Alpoge}\email{levent.alpoge@gmail.com}
\address{Churchill College, University of Cambridge, Cambridge CB3 0DS.}

\begin{document}

\begin{abstract}
Let $L(n)$ be the number of Latin squares of order $n$, and let $L^{\textrm{even}}(n)$ and $L^{\textrm{odd}}(n)$ be the number of even and odd such squares, so that $L(n) = L^{\textrm{even}}(n) + L^{\textrm{odd}}(n)$. The Alon-Tarsi conjecture states that $L^{\textrm{even}}(n)\neq L^{\textrm{odd}}(n)$ when $n$ is even (when $n$ is odd the two are equal for very simple reasons). In this short note we prove that $$|L^{\textrm{even}}(n) - L^{\textrm{odd}}(n)|\leq L(n)^{\frac{1}{2} + o(1)},$$ thus establishing the conjecture that the number of even and odd Latin squares, while conjecturally not equal in even dimensions, are equal to leading order asymptotically. Two proofs are given: both proceed by applying a differential operator to an exponential integral over $\mathrm{SU}(n)$. The method is inspired by a recent result of Kumar-Landsberg.
\end{abstract}

\maketitle

\newtheoremstyle{dotless}{}{}{\itshape}{}{\bfseries}{}{ }{}

\newtheorem{thm}{Theorem}
\newtheorem{lem}[thm]{Lemma}
\newtheorem{remark}[thm]{Remark}
\newtheorem{cor}[thm]{Corollary}
\newtheorem{defn}[thm]{Definition}
\newtheorem{prop}[thm]{Proposition}
\newtheorem{conj}[thm]{Conjecture}
\newtheorem{claim}[thm]{Claim}
\newtheorem{exer}[thm]{Exercise}
\newtheorem{fact}[thm]{Fact}

\theoremstyle{dotless}

\newtheorem{thmnodot}[thm]{Theorem}
\newtheorem{lemnodot}[thm]{Lemma}
\newtheorem{cornodot}[thm]{Corollary}

\newcommand{\image}{\mathop{\text{image}}}
\newcommand{\End}{\mathop{\text{End}}}
\newcommand{\Hom}{\mathop{\text{Hom}}}
\newcommand{\Sum}{\displaystyle\sum\limits}
\newcommand{\Prod}{\displaystyle\prod\limits}
\newcommand{\Tr}{\mathop{\mathrm{Tr}}}
\renewcommand{\Re}{\operatorname{\mathfrak{Re}}}
\renewcommand{\Im}{\operatorname{\mathfrak{Im}}}
\newcommand{\im}{\mathrm{im}\,}
\newcommand{\inner}[1]{\langle #1 \rangle}
\newcommand{\pair}[2]{\langle #1, #2\rangle}
\newcommand{\ppair}[2]{\langle\langle #1, #2\rangle\rangle}
\newcommand{\Pair}[2]{\left[#1, #2\right]}
\newcommand{\Char}{\mathop{\mathrm{char}}}
\newcommand{\rank}{\mathrm{rank}}
\newcommand{\sgn}[1]{\mathop{\mathrm{sgn}}(#1)}
\newcommand{\leg}[2]{\left(\frac{#1}{#2}\right)}
\newcommand{\Sym}{\mathrm{Sym}}
\newcommand{\hmat}[2]{\left(\begin{array}{cc} #1 & #2\\ -\bar{#2} & \bar{#1}\end{array}\right)}
\newcommand{\HMat}[2]{\left(\begin{array}{cc} #1 & #2\\ -\overline{#2} & \overline{#1}\end{array}\right)}
\newcommand{\Sin}[1]{\sin{\left(#1\right)}}
\newcommand{\Cos}[1]{\cos{\left(#1\right)}}
\newcommand{\comm}[2]{\left[#1, #2\right]}
\newcommand{\Isom}{\mathop{\mathrm{Isom}}}
\newcommand{\Map}{\mathop{\mathrm{Map}}}
\newcommand{\Bij}{\mathop{\mathrm{Bij}}}
\newcommand{\Z}{\mathbb{Z}}
\newcommand{\R}{\mathbb{R}}
\newcommand{\Q}{\mathbb{Q}}
\newcommand{\C}{\mathbb{C}}
\newcommand{\Nm}{\mathrm{Nm}}
\newcommand{\RI}[1]{\mathcal{O}_{#1}}
\newcommand{\F}{\mathbb{F}}
\renewcommand{\Pr}{\displaystyle\mathop{\mathrm{Pr}}\limits}
\newcommand{\E}{\mathbb{E}}
\newcommand{\coker}{\mathop{\mathrm{coker}}}
\newcommand{\id}{\mathop{\mathrm{id}}}
\newcommand{\Oplus}{\displaystyle\bigoplus\limits}
\renewcommand{\Cap}{\displaystyle\bigcap\limits}
\renewcommand{\Cup}{\displaystyle\bigcup\limits}
\newcommand{\Bil}{\mathop{\mathrm{Bil}}}
\newcommand{\N}{\mathbb{N}}
\newcommand{\Aut}{\mathop{\mathrm{Aut}}}
\newcommand{\ord}{\mathop{\mathrm{ord}}}
\newcommand{\ch}{\mathop{\mathrm{char}}}
\newcommand{\minpoly}{\mathop{\mathrm{minpoly}}}
\newcommand{\Spec}{\mathop{\mathrm{Spec}}}
\newcommand{\Gal}{\mathop{\mathrm{Gal}}}
\newcommand{\Ad}{\mathop{\mathrm{Ad}}}
\newcommand{\Stab}{\mathop{\mathrm{Stab}}}
\newcommand{\Norm}{\mathop{\mathrm{Norm}}}
\newcommand{\Orb}{\mathop{\mathrm{Orb}}}
\newcommand{\pfrak}{\mathfrak{p}}
\newcommand{\qfrak}{\mathfrak{q}}
\newcommand{\mfrak}{\mathfrak{m}}
\newcommand{\Frac}{\mathop{\mathrm{Frac}}}
\newcommand{\Loc}{\mathop{\mathrm{Loc}}}
\newcommand{\Sat}{\mathop{\mathrm{Sat}}}
\newcommand{\inj}{\hookrightarrow}
\newcommand{\surj}{\twoheadrightarrow}
\newcommand{\bij}{\leftrightarrow}
\newcommand{\Ind}{\mathrm{Ind}}
\newcommand{\Supp}{\mathop{\mathrm{Supp}}}
\newcommand{\Ass}{\mathop{\mathrm{Ass}}}
\newcommand{\Ann}{\mathop{\mathrm{Ann}}}
\newcommand{\Krulldim}{\dim_{\mathrm{Kr}}}
\newcommand{\Avg}{\mathop{\mathrm{Avg}}}
\newcommand{\innerhom}{\underline{\Hom}}
\newcommand{\triv}{\mathop{\mathrm{triv}}}
\newcommand{\Res}{\mathrm{Res}}
\newcommand{\eval}{\mathop{\mathrm{eval}}}
\newcommand{\MC}{\mathop{\mathrm{MC}}}
\newcommand{\Fun}{\mathop{\mathrm{Fun}}}
\newcommand{\InvFun}{\mathop{\mathrm{InvFun}}}
\renewcommand{\ch}{\mathop{\mathrm{ch}}}
\newcommand{\irrep}{\mathop{\mathrm{Irr}}}
\newcommand{\len}{\mathop{\mathrm{len}}}
\newcommand{\SL}{\mathrm{SL}}
\newcommand{\GL}{\mathrm{GL}}
\newcommand{\PSL}{\mathrm{SL}}
\newcommand{\actson}{\curvearrowright}
\renewcommand{\H}{\mathbb{H}}
\newcommand{\mat}[4]{\left(\begin{array}{cc} #1 & #2\\ #3 & #4\end{array}\right)}
\newcommand{\interior}{\mathop{\mathrm{int}}}
\newcommand{\floor}[1]{\left\lfloor #1\right\rfloor}
\newcommand{\iso}{\cong}
\newcommand{\eps}{\epsilon}
\newcommand{\disc}{\mathrm{disc}}
\newcommand{\Frob}{\mathrm{Frob}}
\newcommand{\charpoly}{\mathrm{charpoly}}
\newcommand{\afrak}{\mathfrak{a}}
\newcommand{\cfrak}{\mathfrak{c}}
\newcommand{\codim}{\mathrm{codim}}
\newcommand{\ffrak}{\mathfrak{f}}
\newcommand{\Pfrak}{\mathfrak{P}}
\newcommand{\homcont}{\hom_{\mathrm{cont}}}
\newcommand{\vol}{\mathrm{vol}}
\newcommand{\ofrak}{\mathfrak{o}}
\newcommand{\A}{\mathbb{A}}
\newcommand{\I}{\mathbb{I}}
\newcommand{\invlim}{\varprojlim}
\newcommand{\dirlim}{\varinjlim}
\renewcommand{\ch}{\mathrm{char}}
\newcommand{\artin}[2]{\left(\frac{#1}{#2}\right)}
\newcommand{\Qfrak}{\mathfrak{Q}}
\newcommand{\ur}[1]{#1^{\mathrm{ur}}}
\newcommand{\absnm}{\mathcal{N}}
\newcommand{\ab}[1]{#1^{\mathrm{ab}}}
\newcommand{\G}{\mathbb{G}}
\newcommand{\dfrak}{\mathfrak{d}}
\newcommand{\Bfrak}{\mathfrak{B}}
\renewcommand{\sgn}{\mathrm{sgn}}
\newcommand{\disjcup}{\bigsqcup}
\newcommand{\zfrak}{\mathfrak{z}}
\renewcommand{\Tr}{\mathrm{Tr}}
\newcommand{\reg}{\mathrm{reg}}
\newcommand{\subgrp}{\leq}
\newcommand{\normal}{\vartriangleleft}
\newcommand{\Dfrak}{\mathfrak{D}}
\newcommand{\nvert}{\nmid}
\newcommand{\K}{\mathbb{K}}
\newcommand{\pt}{\mathrm{pt}}
\newcommand{\RP}{\mathbb{RP}}
\newcommand{\CP}{\mathbb{CP}}
\newcommand{\rk}{\mathrm{rk}}
\newcommand{\redH}{\tilde{H}}
\renewcommand{\H}{\tilde{H}}
\newcommand{\Cyl}{\mathrm{Cyl}}
\newcommand{\T}{\mathbb{T}}
\newcommand{\Ab}{\mathrm{Ab}}
\newcommand{\Vect}{\mathrm{Vect}}
\newcommand{\Top}{\mathrm{Top}}
\newcommand{\Nat}{\mathrm{Nat}}
\newcommand{\inc}{\mathrm{inc}}
\newcommand{\Tor}{\mathrm{Tor}}
\newcommand{\Ext}{\mathrm{Ext}}
\newcommand{\fungrpd}{\pi_{\leq 1}}
\newcommand{\slot}{\mbox{---}}
\newcommand{\funct}{\mathcal}
\newcommand{\Funct}{\mathcal{F}}
\newcommand{\Gunct}{\mathcal{G}}
\newcommand{\FunCat}{\mathrm{Funct}}
\newcommand{\Rep}{\mathrm{Rep}}
\newcommand{\Specm}{\mathrm{Specm}}
\newcommand{\ev}{\mathrm{ev}}
\newcommand{\frpt}[1]{\{#1\}}
\newcommand{\h}{\mathscr{H}}
\newcommand{\poly}{\mathrm{poly}}
\newcommand{\Partial}[1]{\frac{\partial}{\partial #1}}
\newcommand{\Cont}{\mathrm{Cont}}
\renewcommand{\o}{\ofrak}
\newcommand{\bfrak}{\mathfrak{b}}
\newcommand{\Cl}{\mathrm{Cl}}
\newcommand{\ceil}[1]{\lceil #1\rceil}
\newcommand{\hfrak}{\mathfrak{h}}
\newcommand{\Sel}{\mathrm{Sel}}
\newcommand{\Qbar}{\overline{\mathbb{Q}}}
\renewcommand{\I}{\mathrm{I}}
\newcommand{\II}{\mathrm{II}}
\newcommand{\III}{\mathrm{III}}
\newcommand{\IV}{\mathrm{IV}}
\newcommand{\V}{\mathrm{V}}
\newcommand{\FuniversalT}{\mathcal{F}_{\mathrm{universal}}^{\leq T}}
\newcommand{\FAT}{\mathcal{F}_{A=0}^{\leq T}}
\newcommand{\FBT}{\mathcal{F}_{B=0}^{\leq T}}
\newcommand{\FcongT}{\mathcal{F}_{\mathrm{congruent}}^{\leq T}}
\newcommand{\rad}{\mathrm{rad}}
\newcommand{\const}{\mathrm{const}}
\renewcommand{\sp}{\mathrm{span}}
\renewcommand{\d}{\partial}
\newcommand{\num}{\mathrm{num}}
\newcommand{\den}{\mathrm{den}}
\newcommand{\ind}{\mathrm{ind}}
\newcommand{\sign}{\mathrm{sign}}
\newcommand{\SU}{\mathrm{SU}}
\newcommand{\U}{\mathrm{U}}
\newcommand{\tr}{\mathrm{tr}}
\newcommand{\diag}{\mathrm{diag}}

\let\uglyphi\phi
\let\phi\varphi


\section{Introduction}
By a \emph{Latin square} we will mean an $n\times n$ matrix containing the numbers $1,\ldots,n$ such that each row and column of the matrix is a permutation of $[n] := \{1,\ldots,n\}$. Note that this is equivalent to a decomposition $$P_{\sigma_1} + \cdots + P_{\sigma_n} = \left(\begin{array}{cccc} 1 & 1 & \cdots & 1\\ \vdots & \vdots & & \vdots\\ 1 & 1 & \cdots & 1\end{array}\right)$$ of the all-ones matrix into a sum of permutation matrices $P_\sigma$ of permutations $\sigma$ (specifically, $\sigma_i$ indicates the positions of $i$), and we will identify a Latin square with its tuple of permutations $\vec{\sigma} := (\sigma_1, \ldots, \sigma_n)$. The \emph{sign} of a Latin square $\vec{\sigma}$ is defined to be $$\sign(\vec{\sigma}) := (-1)^{\frac{n(n-1)}{2}}\prod_{i=1}^n \sign(\sigma_i),$$ where $\sign: S_n\to \{\pm 1\}$ is the usual sign homomorphism. Let $L(n)$ denote the number of Latin squares, and $L^{\textrm{even}}(n)$ and $L^{\textrm{odd}}(n)$ denote the number of even and odd Latin squares, respectively.\footnote{Hence e.g.\ $L(6) = 812851200$, $L^{\textrm{even}}(4) = 576$, and $L^{\textrm{odd}}(5) = 80640$.} Note that $$L^{\textrm{even}}(n) - L^{\textrm{odd}}(n) = (-1)^{\frac{n(n-1)}{2}}\sum_{\vec{\sigma}\text{ Latin}} \sign(\vec{\sigma}),$$ and so $L^{\textrm{even}}(n) - L^{\textrm{odd}}(n)$ is the coefficient of $\prod_{i,j=1}^n X_{ij}$ in the polynomial $(-1)^{\frac{n(n-1)}{2}}\det(X)^n$, where $X := (X_{ij})_{i,j=1}^n$.

Since multiplying each $\sigma_i$ by e.g.\ the transposition $(1\ 2)$ changes $\sign(\vec{\sigma})$ by a factor of $(-1)^n$ but preserves Latinness, if $n$ is odd the number of even and odd Latin squares are equal. The longstanding conjecture of Alon and Tarsi states that this is the only case of equality.

\begin{conj}[Alon-Tarsi \cite{alontarsi}.]
Let $n$ be even. Then $L^{\textrm{even}}(n)\neq L^{\textrm{odd}}(n)$.
\end{conj}

This conjecture implies the Dinitz conjecture (since proven) and the Rota basis conjecture in even dimensions. It has remained open since 1992, though it has been proven for $n$ of the form $p\pm 1$ for an odd prime $p$ \cite{glynn, drisko}. It has been conjectured --- cf.\ e.g.\ \cite{stoneswanless} --- that, though different, the number of even and odd Latin squares should at least agree up to leading order (i.e., the sign of a Latin square should be equidistributed). We prove this in two ways, the more interesting of which proceeds by proving the following identity. To state the identity, we will need the following notation: for a multiindex of nonnegative integers $\alpha\in \N^{[n]\times [n]}$, let $|\alpha| := \sum_{i,j=1}^n \alpha_{ij}$, $X^{\alpha} := \prod_{i,j=1}^n X_{ij}^{\alpha_{ij}}$, and $\alpha! := \prod_{i,j=1}^n \alpha_{ij}!$. The identity then reads as follows.

\begin{thm}\label{identity}
Let $k\geq 0$. For $\alpha$ such that $|\alpha| = kn$, let $c_\alpha$ be the coefficient of $X^\alpha$ in $\det(X)^k$. Then $$\sum_{|\alpha| = kn} c_{\alpha}^2 \alpha! = \frac{k!\cdots (k+n-1)!}{0!\cdots (n-1)!}.$$
\end{thm}

From this we immediately obtain:\footnote{By $f\ll g$ we mean there exists a positive $C > 0$ such that $|f|\leq C |g|$ as functions. By $o(1)$ we mean a quantity tending to $0$ as $n\to\infty$. Note that $2e^{\frac{1}{4}} = 2.56805...$}
\begin{cor}\label{corollary}
$$\left|L^{\textrm{even}}(n) - L^{\textrm{odd}}(n)\right|\leq L(n)^{\frac{1}{2}}\cdot \left(2e^{\frac{1}{4}}\right)^{n^2(1+o(1))}.$$
\end{cor}

\begin{proof}[Proof of Corollary \ref{corollary}.]
This follows from partial summation and the theorem of van Lint-Wilson counting the total number of Latin squares, which we state here for convenience.

\begin{thm}[Theorem 17.3 in van Lint-Wilson \cite{vanlintwilson}.]
We have the following asymptotic for the number of Latin squares of order $n$:
$$L(n) = n^{n^2} e^{-2n^2(1 + o(1))}.$$
\end{thm}

Specifically, given this it suffices to show that $$c_{(1,\ldots,1)}^2\leq n^{n^2}\left(\frac{4}{e^{\frac{3}{2}}}\right)^{n^2(1+o(1))},$$ where $(1,\ldots,1)$ is the all-ones multiindex. Since the left-hand side of Theorem \ref{identity} with $k=n$ is certainly at least this large, it suffices to show that $$\log{\left(\frac{n!\cdots (2n-1)!}{0!\cdots (n-1)!}\right)} = \sum_{a=1}^n\sum_{b=a}^{a+n-1} \log{b} = n^2\log{n} - n^2\left(\frac{3}{2} - \log{4}\right) + o(n^2),$$ which follows via partial summation.
\end{proof}

In fact we can improve the exponential factor ever so slightly, obtaining:\footnote{$\frac{4}{\sqrt{e}} = 2.42612...$ Note that we recover the equivalence, noted in \cite{kumarlandsberg}, of the Alon-Tarsi conjecture and the nonvanishing of $$\int_{\SU(n)} U_{11}\cdots U_{nn} dU.$$}
\begin{cor}\label{better corollary}
If $|\alpha| = kn$, then $$c_\alpha \alpha! = \frac{k!\cdots (k+n-1)!}{0!\cdots (n-1)!}\int_{\SU(n)} U^\alpha dU.$$ It follows that
$$\left|L^{\textrm{even}}(n) - L^{\textrm{odd}}(n)\right|\leq L(n)^{\frac{1}{2}}\cdot \left(\frac{4}{\sqrt{e}}\right)^{n^2(1+o(1))}.$$
\end{cor}

In any case, having shown that Theorem \ref{identity} implies square-root cancellation of the sign over Latin squares, let us now prove Theorem \ref{identity}. The method we use will then allow us to improve the resulting bound slightly to get Corollary \ref{better corollary}. (In any case, these bounds are likely far from optimal, though there is not enough numerical data available to make a precise conjecture.)

\section{Proof of Theorem \ref{identity}}

\begin{proof}[Proof of Theorem \ref{identity}.]
Consider the expression $$\det(\d)^k\int_{\SU(n)} e^{\tr(XU)} dU\vert_{X=0},$$ where $dU$ denotes normalized Haar measure on $\SU(n)$, and $\det(\d)$ is the differential operator obtained by substituting $\d_{ij} := \frac{\d}{\d X_{ij}}$ into the expression for the determinant of a matrix. On the one hand, this is simply $$\int_{\SU(n)} \det(U)^k e^{\tr(XU)} dU\vert_{X=0} = 1.$$ On the other hand, as in \cite{creutz}, by invariance of the integral under changes of variable $X\mapsto UXV$ with $U,V\in \SU(n)$, we may first replace $X$ by $(X^\dagger X)^{\frac{1}{2}}$ via $$X = (X^\dagger X)^{\frac{1}{2}} [(X^\dagger X)^{-\frac{1}{2}} X]$$ (the term in brackets is unitary\footnote{Here we approach $X = 0$ through invertible matrices $X$.}), and then by the spectral theorem we reduce to the case of $X$ diagonal. But then by changing variables via $X\mapsto X\cdot \mathrm{diag}(1,\ldots,\zeta,\ldots,\zeta^{-1},\ldots 1)$ with $\zeta\in S^1$, it follows that terms not containing the eigenvalues of $X$ with the same multiplicity automatically vanish. Hence we see that $$\int_{\SU(n)} e^{\tr(XU)} dU = \sum_{\ell\geq 0} \frac{\det(X)^\ell}{(\ell n)!} \int_{\SU(n)} \tr(U)^{\ell n} dU.$$ The latter integrals are simply the multiplicity of the trivial representation in $V^{\otimes \ell n}$, where $V$ is the fundamental $n$-dimensional representation of $\SU(n)$. By Schur-Weyl duality, this multiplicity is simply the dimension of the irreducible representation of $S_{\ell n}$ corresponding to the partition $(\ell, \ldots, \ell)$ with precisely $n$ parts. By the hook-length formula this dimension is $$(\ell n)!\frac{0!\cdots (n-1)!}{\ell!\cdots (\ell + n - 1)!}.$$ Since $\det(\d)^k$ is homogeneous of degree $kn$, the only term that matters is the one involving $\det(X)^k$, and so we obtain:
$$1 = \frac{0!\cdots (n-1)!}{\ell!\cdots (\ell + n - 1)!} \det(\d)^k \det(X)^k.$$ But now the result follows simply by expanding out and noting that $\d^\alpha X^\beta = \delta_{\alpha,\beta} \alpha!$.
\end{proof}

\section{Proof of Corollary \ref{better corollary}}

\begin{proof}[Proof of Corollary \ref{better corollary}.]
Instead consider $$\d^\alpha \int_{\SU(n)} e^{\tr(XU)} dU\vert_{X=0}.$$ On the one hand, it is $$\int_{\SU(n)} U^\alpha dU,$$ and on the other hand it is, as noted above, $$c_\alpha \alpha!\cdot \frac{0!\cdots (n-1)!}{k!\cdots (k+n-1)!}.$$ Hence $$c_\alpha \alpha! = \frac{k!\cdots (k+n-1)!}{0!\cdots (n-1)!}\int_{\SU(n)} U^\alpha dU.$$ For the second part, taking $\alpha = (1,\ldots,1)$, by the arithmetic mean-geometric mean inequality, the absolute value of the right-hand side is bounded above by $$\frac{n!\cdots (2n-1)!}{0!\cdots (n-1)!}\int_{\SU(n)} \left(\frac{|U_{11}|^2 + \cdots + |U_{nn}|^2}{n^2}\right)^{\frac{n^2}{2}} dU = \frac{n!\cdots (2n-1)!}{0!\cdots (n-1)!}\cdot n^{-\frac{n^2}{2}}.$$ This completes the proof.
\end{proof}

Finally, let us also note that by combining Theorem \ref{identity} and Corollary \ref{better corollary} one gets a bound on $\int_{\SU(n)} U^\alpha dU$ for all such $\alpha$. Specifically, writing $|\alpha| =: kn$, one gets $$\left|\int_{\SU(n)} U^\alpha dU\right|\leq \left(\frac{0!\cdots (n-1)!}{k!\cdots (k+n-1)!}\cdot \alpha!\right)^{\frac{1}{2}}.$$ (Note that the integral is zero if either $|\alpha|$ is not a multiple of $n$ or if $\frac{n}{|\alpha|}\alpha$ is not doubly-stochastic.)

\section{Acknowledgements}

I would like to thank Qiaochu Yuan and Steve Huntsman for helpful comments on a MathOverflow question that led to this note, and Jacob Tsimerman for reading through an earlier version of this note and for helpful remarks.

\bibliography{squarerootcancellationforthesignsoflatinsquares}{}
\bibliographystyle{plain}

\ \\

\end{document}